\documentclass{amsart}

\usepackage{amsthm}
\usepackage{amscd} 
\usepackage{amssymb}  
\usepackage{tikz}

\allowdisplaybreaks

\numberwithin{equation}{section} 

\newtheorem{Thm}{Theorem}[section] 
\newtheorem{Prop}[Thm]{Proposition} 
\newtheorem{Lem}[Thm]{Lemma} 
\newtheorem{Cor}[Thm]{Corollary} 
 
\theoremstyle{remark}

\theoremstyle{definition} 
\newtheorem{Def}[Thm]{Definition} 
\newtheorem{Exa}[Thm]{Example}

\newcommand{\R}{\mathbb{R}}

\newcommand{\N}{\mathbb{N}}
\newcommand{\Q}{\mathbb{Q}}

\newcommand{\F}{\mathbb{F}}

\newcommand{\mF}{\mathcal{F}}

\newcommand{\ot}{\otimes}

\newcommand{\bit}{\begin{itemize}}
\newcommand{\eit}{\end{itemize}}

\author{Paolo Salvatore}
\address{Dipartimento di Matematica,  
Universit\`{a} di Roma Tor Vergata, 
Via della Ricerca Scientifica, 
00133 Roma, Italy}
\email{salvator@mat.uniroma2.it}
\thanks{The author acknowledges the MIUR Excellence Department Project awarded to
the Department of Mathematics, University of Rome Tor Vergata, CUP E83C18000100006}
\title[Planar non-formality of the little discs]{Planar non-formality of the little discs operad in characteristic two}
\begin{document}

\begin{abstract}
We show that the little discs operad $D_2$ is not formal over $\F_2$ as a planar (or non-symmetric) operad. We compute explicitly the homological obstruction 
using as chain model the cells of the spineless cacti operad. 
\end{abstract}

\maketitle

\section{Introduction}
The little discs operads $D_n$ have been heavily studied since their introduction by Boardman-Vogt in the seventies.
A major result states that they are formal over a field of characteristic 0 as symmetric operads.
This was proved by Kontsevich \cite{Kont} and Lambrechts-Volic \cite{LV} over $\R$, by Tamarkin over $\Q$ for $n=2$ \cite{Tam}, and later by 
Fresse and Willwacher over $\Q$ for $n>2$ \cite{FW}. We recall that a topological operad $O$ is formal over a ring $R$ if  
its homology operad $H_*(O,R)$ and its 
singular chain operad $C_*(O,R)$ are connected by a zig-zag of quasi-isomorphisms, i.e. chain operad maps 
inducing an isomorphism in homology.
It is not difficult to see that $D_n$ cannot be formal over a field $\F$ of positive characteristic as a symmetric operad ( Remark 6.9 in \cite{CH}). 
However if we forget the action of the symmetric groups, and we consider $D_n$ as  a planar (or non-symmetric) operad, then
the weaker notion of planar formality over $\F$ is much harder to check. A reason to study it is the relation to knot theory:
there is a vaste literature relating the space $K_n$ of long knots in $\R^n$ to $D_n$. \cite{Sinha}
For example if the pair $(D_n,D_1)$ was a formal
pair of planar operads, then the Sinha-Vassiliev spectral sequence computing the homology of $K_n$ (for $n>3$) would collapse. This  is a 20 year old
conjecture by Vassiliev that has been verified in characteristic 0 using the rational formality of $D_n$ \cite{0703649}. A related result 
by Turchin and Willwacher \cite{TW} states that rationally the pair $(D_{n},D_{n-k})$ is formal for $k \geq 2$ and not formal for $k=1$. 
In this paper we show:
\begin{Thm} \label{main}
The little 2-discs operad $D_2$  is {\em not} formal as a planar operad over $\F_2$.
\end{Thm}
An immediate consequence is that $D_2$ is not formal over any field of characteristic 2, and also over the integers. 
We prove Theorem \ref{main} by obstruction theory, adapting the work of Halperin-Stasheff to the framework of planar operads. 
We use a finite dimensional model for the chain operad of the little 2-discs, that is basically the cell complex of the spineless cacti operad (\cite{AGT}, sect. 4), or equivalently the second filtration of the surjection operad (1.2 in \cite{BF} ). 
It should be mentioned that our theorem does not disprove Vassiliev's conjecture.  
 The non-formality of $D_2$ as a planar operad, together with the non-formality of the single spaces $D_2(k)$ in characteristic 2 for $k>3$, that we proved in
\cite{formality2}, show that the behaviour in characteristic 2 is opposite to that in 
characteristic 0, where instead Hopf operad formality holds \cite{FW}
(both at the level of spaces and operads). We plan to investigate in future work the formality of $D_n$ for $n>2$, and the case of odd characteristic.
I am grateful to Joana Cirici and Benoit Fresse for some fruitful discussions.

\section{Obstruction theory for planar operads}

The obstruction theory by  Halperin-Stasheff \cite{HS} for rational commutative algebras can be translated almost verbatim to the framework of planar operads. 
The rational version for symmetric operads was worked out by Markl \cite{Markl}. 
We recall the definition of a planar operad in a concrete symmetric monoidal category, that in practice will be 
either the category $Top$ of topological spaces with the cartesian product, or the category $Ch_\F$ of $\N$-graded chain complexes over
a field $\F$ with the tensor product.

\begin{Def} 
A planar operad in a concrete symmetric monoidal category $(\mathcal{C},\ot)$
is a sequence of objects $O(k)$ of $\mathcal{C}$,  with $k \in \N$ (  $k$ is called the arity),

 together with
an element $\iota \in  O(1)$ called the unit, and composition maps
$$\_\circ_i\_:O(k)\ot O(l) \to O(k+l-1) \quad 1\leq i \leq k$$
such that
for $a \in O(n), \, b \in O(p), \,c \in O(q)\, $ 
$$(a \circ_i b)\circ_{j+p-1} c = (a \circ_j c)\circ_i b
 \quad {\rm for }\quad   1 \leq i <j \leq n$$  
 $$a \circ_i (b \circ_j c)=(a \circ_i b) \circ_{i+j-1} c 
 \quad{\rm for}\quad  1 \leq i \leq n,\, 1 \leq j \leq p $$
and $\iota$ is a bi-sided unit for the $\circ_i$-operations.
\end{Def}
A symmetric operad $O$ is a planar operad equipped with the action of the symmetric group $\Sigma_k$ on $O(k)$ for each $k$,
so that the actions and the composition maps are compatible in an appropriate sense.  
A planar operad in $Top$ is called a planar topological operad, and an operad in $Ch_\F$ is called 
a differential graded planar operad (DGPO) over $\F$.  
A DGPO equipped with a trivial differential is called a graded planar operad (GPO). A GPO  $O$ is reduced if  
$O(0)=0$ and $O(1)=\F\{\iota\}$.

We sketch the construction of the bigraded model of a GPO.
Let $\mF$ be the free functor from the category of sequences of graded vector spaces $V=\{V(i)\}_{i \geq 2}$ to the category of reduced graded planar operads, that is 
left adjoint to the forgetful functor. Roughly speaking the free functor is constructed by means of directed planar trees with vertices labelled by elements of $V$ (compare Def. 6 in \cite{Lie}). There is a splitting $$\mF(V)=\mF^+(V) \oplus V \oplus \F\{\iota\},$$ where $\mF^+(V)$ contains the decomposable elements. 

\begin{Def}
A bigraded DGPO is one of the form $(\mF(W),d)$, where $$W=(W(k))_{k \geq 2}= (\oplus_{i,j}W^i_j(k)  )_{k \geq 2}$$ is a sequence of bigraded vector spaces,
We say that $W^i_j(k)$ contains the elements of arity $k$, dimension $j$ , and level $i$. The dimension is the grading of the underlying chain complex, 
and the level is an additional grading. We denote $W^i = \oplus_{j,k} W_j^i(k),\, W^{\leq i} =\oplus_{i' \leq i} W^{i'}$, and
$W(\leq k)= \oplus_{k' \leq k} W(k')$.
By adding up degrees  $\mF(W)(k)$ inherits a bigrading for each $k$. We require $d$ to be homogeneous of bidegree $(-1,-1)$ with respect to 
the dimension and the level grading. 
\end{Def}

We can consider the homology with respect to the level grading obtaining for each $i$ a GPO $H^i(\mF(W),d)$.

\begin{Prop} 
Let $H$ be a reduced GPO with a presentation $H=\mF(V)/(R)$, where $V=\{V(i)\}_{i \geq 2}$ is a sequence of graded vector spaces,
$R \subset \mF(V)$ is a subsequence of graded vector space, and $(R)$ is the operadic ideal generated by $R$.
Then there is a bigraded DGPO $(\mF(W),d)$ and 
a quasi-isomorphism $$\rho: (\mF(W),d) \to (H,0)$$   of DGPO such that 
\bit
\item $W^0 = V$, 
\item $\rho_{| \,V}: V \to H$ is the inclusion
\item $\rho_{| \, W^i}:W^i \to H$ is trivial for $i>0$ 
\item $H^i(\mF(W),d)=0$ for $i>0$ 	 and $H^0(\mF(W),d) \stackrel{\rho_*}{\cong} H$
\item $d(x) \in \mF^+(W)$ for $x \in W$ (the differential is decomposable)
\eit
Under these conditions $(\mF(W),d)$ is uniquely defined up to isomorphism, and is called the {\em bigraded model}, or the {\em minimal model} of $H$.
In particular $W^1={\rm s}R$ is the suspended sequence of relations, where {\rm s} raises the dimension degree by one.
In general  $$W^i={\rm s}(H^{i-1}(\mF(W^{\leq (i-1)},d)  )  )/H$$ is the suspended module of $H$-bimodule-indecomposables
and $$d:W^i \to \mF(W^{\leq (i-1)})$$ is a splitting of the projection.  
\end{Prop}

\begin{Def}
Let $O(k)=O^i_j(k)$ be a sequence of bigraded vector spaces.
We say that a linear map $f:O \to O$ lowers the filtration level by  $k$ 
if $f(O^i) \subseteq O^{\leq(i-k)}$ for each $i$.
\end{Def}


\begin{Prop} 
Let $C$ be a DGPO such that $H = H_*(C)$ is a reduced GPO. Let $\rho:(\mF(W),d) \to (H,0)$ be the bigraded model of $H$.
Then there is a differential $D$ such that $(\mF(W),D)$ is a DGPO, called the {\em filtered model} of $C$,  
together with a quasi-isomorphism $\pi: (\mF(W),D) \to C$  such that
\bit
\item $D-d$ lowers the filtration level by 2
\item  for $x \in W^0$ 
$\pi_*(\bar{x}) = \rho(x) \in H_*(C)=H$.
\eit
If $\pi':(\mF(W),D') \to C$ is another filtered model, then 
there exists an isomorphism $\Phi: (\mF(W),D) \cong (\mF(W),D')$ such that $\Phi-id$ lowers the filtration level by 1.
\end{Prop}

\begin{Def}
Two DGPO $C_1$ and $C_2$ are weakly equivalent if there is a zig-zag of quasi-isomorphisms 
$C_1 \leftarrow \dots \rightarrow C_2$ connecting them.
\end{Def}
We notice that a zig-zag of equivalences induces an isomorphism in homology $H_*(C_1) \cong H_*(C_2)$.

\begin{Prop} \label{equi}
Let $C_1$ and $C_2$ be DGPO with $H=H_*(C_1)=H_*(C_2)$ reduced, and with respective filtered models $(\mF(W),D_1), \;
(\mF(W),D_2)$. Then $C_1$ and $C_2$ are weakly equivalent by a zig-zag inducing the identity in homology
if and only if there is an isomorphism $\Phi:(\mF(W),D_1) \to (\mF(W),D_2)$ 
with $\Phi-id$ lowering the filtration level by 1.
\end{Prop}

\begin{Def}
A DGPO $C$ is formal if it is weakly equivalent to its homology $H_*(C)$ equipped with the trivial differential.
A topological planar operad $O$ is formal over $\F$ if the singular chain DGPO $C_*(O,\F)$ is formal. 
\end{Def}
From Proposition \ref{equi} we obtain the following corollary.
\begin{Cor} \label{obst}
A DGPO $C$ is formal if and only if there is an isomorphism $$\Phi: (\mF(W),d) \cong (\mF(W),D)$$ between the bigraded model of $H_*(C)$
and a filtered model of $C$, with $\Phi-id$ lowering the filtration level by 1.
\end{Cor}

Let $C$ be a DGPO such that $H_*(C)$ is reduced, and has the bigraded model $(\mF(W),d)$. 
We perform a partial construction of the filtered model of $C$ up to level 2 and define the first obstruction to the formality of $C$.

We start by defining up to level 1 $$\pi: \mF(W^{\leq 1}, d) \to C$$ by sending linear generators $x \in W^0$ to cycles $\pi(x) \in Z(C)$ representing 

$\rho(x) \in H_*(C)$,
and relations $y \in W^1$ to elements $\pi(y)$ such that 

$\pi(d(y))=d_C(\pi(y)) \in B(C)$ is a boundary representing the relation $y$
in the homology $H_*(C)=\mF(W^0)/(dW^1)$.
Now for $z \in W^2$ consider $a(z) :=  \pi(dz)$. Since $\pi$ commutes with the differentials, $d_C(\pi(dz))=\pi(ddz)=0$, and we can consider the homology
class  $\alpha(z) = \overline{a(z)} \in H_*(C)$.

We have that $\alpha  \in Hom_{-1}(W^2,H_*(C))$, where the latter is the vector space
of arity preserving linear maps that lower the dimension by 1.
Clearly $\pi$ can be extended to a DGPO-map up to level 2 
$$\pi:(\mF(W^{\leq 2}),d) \to (C,d_C)$$ if and only if $\alpha=0$.
In that case $D=d$ on $W^{\leq 2}$. Otherwise 
let $\eta:H_*(C) \to \mF(W^0)$ be a linear splitting of the projection. 
We define, on $W^2$, $D=d-\eta \circ \alpha$. 
\begin{Prop}
There exists a DGPO-map $\pi:(\mF(W^{\leq 2}),D) \to C$ commuting with the differentials.
\end{Prop}
\begin{proof}
For any linear generator $z \in W^2$ by construction the cycle $$\pi( D(z))=\pi(d(z))-\pi (\eta( \alpha(z)))$$ represents the homology class $\alpha(z)-\alpha(z)=0$
and so there exists $c  \in C$ such that $d_C(c)=\pi(D(z))$. Set $\pi(z)=c$.
\end{proof}
This procedure can be continued as in the proof of Theorem 4.4 in  \cite{HS} to obtain the filtered model of $C$. 
We are now concerned about the formality of $C$.

\begin{Def} \label{part}
Let $Hom_0(W^1 , H_*(C))$ be the vector space of linear maps preserving both the arity and the dimension.
  There is a homomorphism
$$\partial: Hom_0(W^1 , H_*(C)) \rightarrow Hom_{-1}(W^2,H_*(C))$$
defined as follows: for  $f:W^{1} \to H_*(C)$, extend it to a linear map $f:W^{\leq 1}=W^0 \oplus W^1 \to H_*(C)$
by $f_{| \,W^0}=\rho$. By the universal property $f$ defines an operad map
$\hat{f}:\mF(W^{\leq 1}) \to H_*(C)$. 
Then $\partial(f)$ is the composition
$$\partial(f):W^2 \stackrel{d}{\rightarrow} \mF(W^{\leq 1}) \stackrel{\hat{f}}{\rightarrow} H_*(C)$$  
\end{Def}

\begin{Prop}
There exists an isomorphism $\Phi:(\mF(W^{\leq 2}),d) \to (\mF(W^{\leq 2}),D)$ such that $\Phi-id$ lowers the filtration by 1 if and only if 
$\alpha \in Im(\partial)$.
\end{Prop}

The proof is similar to that of Lemma 6.7 in \cite{formality2}.

\begin{Cor} \label{nonformal}
If $\alpha \notin Im(\partial)$ then $C$ is not formal. 
\end{Cor}
Thus $\alpha$ is the first obstruction to the formality of $C$.

\section{The bigraded model of the homology of the little discs operad}

Let us describe the first generators of the bigraded model of the homology of the little 2-discs $H=H_*(D_2,\F_2)$, as planar operad, over 
$\F_2$.  We will consider the non-unitary version of the little discs $D_2$ with $D_2(0)=\emptyset$.

It is well known that $H$ is the Gerstenhaber operad, and has the following presentation as a {\em symmetric} operad.

\begin{Prop}
The Gerstenhaber operad $H$, as a symmetric operad, is generated by
\bit
\item the product
$m=m(x_1,x_2)=x_1 x_2$ in degree $0$ and arity $2$, 
\item  the bracket $b=b(x_1,x_2)=[x_1,x_2]$ in degree $1$ and arity $2$
\eit
modulo the  
following relations expressed using the module structure of $H_i$ over the group ring of the symmetric group $\F_2[\Sigma_i]$.
On the right we express the corresponding relations holding in Gerstenhaber algebras. 
\bit
\item the associativity relation 
$$m \circ_1 m = m \circ_2 m   \quad \quad  x_1 (x_2 x_3)=(x_1 x_2) x_3$$     
\item the commutativity relation for the product 
$$(21) m = m     \quad  \quad  x_1x_2 = x_2 x_1 $$
\item the (anti)commutativity relation for the bracket 
$$(21)b=b    \quad   \quad [x_1,x_2]=[x_2,x_1]$$
\item the Jacobi relation  
$$((123)+(231)+(312)) (b \circ_1 b ) =0  \quad [x_1,[x_2,x_3]]+[x_2,[x_3,x_1]]+[x_3,[x_1,x_2]] = 0$$
\item The Poisson relation 
$$b \circ_1 m = ((123)+ (213)) (m \circ_2 b) \quad  \quad  [x_1x_2,x_3]=x_1[x_2,x_3]+x_2[x_1,x_3]$$
\eit
\end{Prop}
\medskip
The Poincar\'e polynomials $P_k(t)=\sum_i \dim H_i(k)t^i$ in arity $k=2,3,4$ are
$$P_2(t)=1+t, \, P_3(t)=1+3t+2t^2, \,  P_4(t)=1+6t+11t^2+6t^3$$
We list basis generators 
\begin{align*}
&H_0(2)=\F_2 \{x_1x_2\} \\
&H_1(2)=\F_2\{ [x_1,x_2]\} \\
&H_0(3)=\F_2\{x_1 x_2 x_3\} \\
&H_1(3)=\F_2 \{  [x_1,x_2]x_3, \,  [x_1,x_3]x_2, \, [x_2,x_3]x_1 \} \\
&H_2(3)=\F_2\{  [[x_1,x_2],x_3], \, [x_1,[x_2,x_3]] \}  \\
&H_0(4)= \F_2\{ x_1 x_2 x_3 x_4  \}   \\
&H_1(4)=\F_2 \{   [x_1,x_2]x_3x_4, \, [x_1,x_3]x_2 x_4,\, [x_1,x_4]x_2 x_3 , \, [x_2,x_3]x_1x_4, \, [x_2,x_4]x_1x_3, \, [x_3,x_4]x_1x_2 \} \\
&H_2(4)=\F_2\{ [[x_1,x_2],x_3]x_4, \, [x_1,[x_2,x_3]]x_4,\, [[x_2,x_3],x_4]x_1, \, [x_2,[x_3,x_4]]x_1, \,  [x_1,[x_3,x_4]]x_2, \, \\
&[x_1,[x_3,x_4]]x_2, \, [x_1,[x_2,x_4]]x_3, \, [[x_1,x_2],x_4]x_3, \,   [x_1,x_2][x_3,x_4],\, [x_1,x_3][x_2,x_4],  \,[x_1,x_4][x_2,x_3] \} \\
&H_3(4)=\F_2\{ [x_1,[x_2,[x_3,x_4]]], \,   [x_1,[[x_2,x_4],x_3]], \,  [[x_1,x_4],[x_2,x_3]], \, [[x_1,x_3],[x_2,x_4]], \\
& \, [[x_1,[x_3,x_4]],x_2] , \, [[[x_1,x_4],x_3],x_2] \}  
\end{align*}
A geometric description of the cycles is given in \cite{palermo}.

The presentation of $H_*(D_2)$ as a planar operad is very different from the symmetric presentation.
For example consider the symmetric sub-operad of $H$ generated by the bracket $b$: up to one operadic suspension this is the Lie operad  $Lie$,
that we studied as a planar operad in \cite{Lie}.
\begin{Thm} \cite{Lie}
The operad $Lie$ is a free 
planar operad on an infinite number of generators growing exponentially in the arity.
\end{Thm}
 The next planar generator of $Lie$ after the bracket is
$l=(1324) ((b \circ_2 b) \circ_1 b)$ corresponding to the arity 4 operation $[[x_1,x_3],[x_2,x_4]]$.

\

We compute all generators of the bigraded model of $H_*(D_2)$ up to arity 4.
\begin{Thm}
The bigraded model $(\mF(W),d) \to H=H_*(D_2,\F_2)$ of the Gerstenhaber operad over $\F_2$ 
satisfies 
$$W^0(\leq 4)=\F_2\{m,b,u,l\},  \, W^1(\leq 4)=\F_2\{A,B\} , \, W^2(\leq 4)=\F_2\{P,C\}$$ 
In arity 2, level 0, we have.
\bit
\item the product $m$ in dimension $0$ 
\item  the bracket $b$ in dimension $1$. 
\eit
In arity 3, level 1, we have 
\bit
\item a 1-dimensional generator $A$ resolving associativity such that $$d(A)=m\circ_1 m +m \circ_2 m,$$
\item a 2-dimensional generator $B$ resolving the Poisson relation such that 
$$d(B)=b\circ_1 m + b\circ_2 m + m\circ_1 b + m \circ_2 b \, .$$
\item In arity 4, level 0 , dimension 2, there is a generator 
$$u=(1324) ((m \circ_2 b) \circ_1 b) = [x_1,x_3][x_2,x_4] $$ 
\item In arity 4, level 0, dimension 3, we have the generator presented earlier
$$l=(1324) ((b \circ_2 b) \circ_1 b) = [[x_1,x_3],[x_2,x_4]] $$
\item In arity 4, level 2, dimension 2, we have a generator $P$ resolving the Pentagon relation 
$$d(P)=m\circ_1 A + m\circ_2 A+ A \circ_1 m + A \circ_2 m + A \circ_3 m \, .$$
\item Finally in arity 4, level 2, dimension 3, there is a generator $C$ with 
$$d(C)=A\circ_1 b + A \circ_2 b + A\circ_3 b + b\circ_1 A + b\circ_2 A+B\circ_1 m+B \circ_2 m+B \circ_3 m+m\circ_1 B+m\circ_2 B\, .$$
\eit
\end{Thm}
\begin{proof}
The product and the bracket generate under operad composition all of $H(3)$, all of $H_i(4)$ for $i \leq 1$, a subspace of $H_2(4)$ of codimension 1 not 
containing $u$, and a subspace of codimension 1 of $H_3(4)$ not containing $l$.
In arity 3 and dimension 0 
$$\mF(W^0)_0(3)=\{m \circ_1 m, m \circ_2 m \}$$ and the kernel  of $\rho:W_0^0(3) \to H_0(3)\cong \F_2$ is generated by $dA$.
In arity 3 and dimension 1  $$\mF(W^0)_1(3)= \{m \circ_1 b, m \circ_2 b, b \circ_1 m, b \circ_2 m\}$$ and the kernel of $\rho:W_1^0(3) \to H_1(3)\cong (\F_2)^3$ is generated 
by $dB$. In arity 3 and dimension 2 $$\mF(W^0)_2(3)=\{ b \circ_1 b, b \circ_2 b\} \cong H_2(3).$$ 
In arity 4 and dimension 1 $$\mF(W^1)_1(4) \cong (\F_2)^5$$ and $$d:\mF(W^1)_1(4) \to \mF(W^0)_0(4) \cong (\F_2)^5$$ has rank 4, with 
cokernel $H_0(4)$ and kernel generated by $dP$.
In arity 4 and dimension 2 $$\mF(W^1)_2(4) \cong (\F_2)^{10}$$ and $$d:\mF(W^1)_2(4) \to \mF(W^0)_1(4) \cong (\F_2)^{15}$$  
has rank 9 with cokernel $H_1(4)$ and kernel generated by $dC$.
In arity 4 and dimension 3 $$\mF(W^1)_3(4) \cong (\F_2)^{5}$$ and  
$$d:\mF(W^1)_3(4) \to \mF(W^0)_2(4) \cong (\F_2)^{15}$$   
has rank 5 with cokernel the codimension 1 subspace of $H_2(4)$ generated under operad composition by $m$ and $b$, and trivial kernel.
\end{proof}
\section{The cacti model for the little discs}

We consider the 2nd filtration $S$ of the surjection operad as a model for the chain operad of the little 2-discs operad. 
The operad $S$ is the cellular chain complex of the spineless cacti operad as explained in \cite{AGT}.
For any $k$  $S(k)$ is a free $\F_2[\Sigma_k]$-module.
\begin{Def} \label{cact}
The linear generators of $S(k)$ of dimension $i$ are sequences of length $i+k$ containing all integers $1,2,\dots, k$  such that
\bit
\item No adjacent numbers are equal
\item No ordered sub-sequence of the form $a b a b$ with $a \neq b$ occurs.  
\eit
\end{Def}
The symmetric group $\Sigma_k$ acts on $S(k)$ by acting on the values of sequences.

\begin{Exa} 

\bit
\item $S_0(2)$ has the $\F_2$-basis $\{12,21\}$  
\item $S_1(2)$ has the $\F_2$-basis $\{121,212\}$ 
\item $S_0(3)$ is the free $\F_2[\Sigma_3]$-module on $\{123\}$
\item $S_1(3)$ is the free  $\F_2[\Sigma_3]$-module on $\{1231,1213, 1232\}$
\item $S_2(3)$ is the free  $\F_2[\Sigma_3]$-module on $\{ 12321,12131\}$ 
\eit
\end{Exa}
The top dimensional generators of $S(k)$ are in dimension $k-1$.

The differential $\delta$ of $S$ is obtained by removing an element from a sequence in all possible ways, and adding the results, deleting those sequences 
that do not satisfy the conditions of definition  \ref{cact}.

For example $$\delta(12321)=2321 + 1321 + 1231 + 1232$$ since $1221$ is not allowed.

In order to define the operad composition of $S$ we need the following definition.
\begin{Def}
An interval decomposition of a sequence $y$ into $m$ subsequences is the overlapping partition of $y$ into $m$ ordered non-empty subsequences of consecutive elements $y_1,\dots,y_m$ of $y$ such that  two adjacent subsequences $y_i,y_{i+1}$ have in common exactly the last element of $y_i$ and the first element of $y_{i+1}$.
\end{Def}
For example the interval decompositions of $123$ into three subsequences are 
$$\{1,12,23\},\{1,1,123\},\{1,123,3\},\{123,3,3\},\{12,2,23\},\{12,23,3\}.$$

\begin{Def}
A composition of linear generators of $S$ of the form $$x \circ_i y \in  S(p+q-1),$$ with $x \in S(p)$ and $y \in S(q)$  is obtained by the following procedure: 
let $n$ be the number of occurrences of the value $i$ in $x$.

For each interval decomposition of $y$ into $n$ subsequences $y_1,\dots,y_n$
let $y'_j$ be the sequence obtained by adding $i-1$ to each value of $y_j$.

Add $q-1$ to each value in $x$ larger than $i$.

Replace  the $j$-th occurrence of $i$ in $x$ by the sequence $y'_j$ for each $j=1,\dots,n$.

Take the sum of the resulting sequences over all interval decompositions of $y$ into $n$ subsequences.
\end{Def}

For example $$2123 \circ_2 121 = (2)1(232)4+(23)1(32)4+(232)1(2)4$$ 
where we added some parenthesis to emphasize the subsequences $y'_j$. 

\begin{Thm} {\rm (3.5 in  \cite{MS})}
The operad $S$ is weakly equivalent to the chain operad of the little 2-discs operad $C_*(D_2,\F_2)$.
\end{Thm}
By this theorem the formality of $S$ is equivalent to the formality of $D_2$ over $\F_2$.
Under the identification $H_*(S) \cong H_*(D_2) \cong H$
\bit
\item the product $m$ is represented by the
cycle $12 \in S_0(2)$ 
\item  the bracket $b$ is represented by the cycle $121+212 \in S_1(2)$
\eit

\section{Computing the obstruction}

We construct a DGPO homomorphism as in section 2  
$$\pi:\mF(W^{\leq 1}(\leq 4)) \to S$$
 in arity $\leq 4$ and level $\leq 1$.
The choice
\begin{align*}
& \pi(m)=12 \\
& \pi(b)=121+212 \\
& \pi(A)=0 \\
& \pi(B)=21312+23132+12131+31323 
\end{align*}
is compatible with the differential on $A$ because
 the multiplication is strictly associative in $S$, i.e.  $(12)\circ_1 (12)=123=(12) \circ_2 (12)$, and 
 so $\pi(dA)=\pi(m \circ_1 + m \circ_2 m)=(12)\circ_1(12)+(12)\circ_2(12)=0=\delta(0)=\delta(\pi(A))$.

It  is also compatible with the differential on $B$ since 
\begin{align*}
&\delta(\pi(B))=(1312+2312+2132+2131)+(3132+2132+2312+2313)+\\
&(2131+1231+1213)+(1323+3123+3132) = \\
&(1232+1312+3123)+ (1231+2313+2123)+ (1213+2123)+ (1232+1323)=  \\
& (121+212)\circ_1 (12)+(121+212)\circ_2 (12)+(12)\circ_1 (121+212)+12 \circ_2 (121+212) =\\
&\pi(b\circ_1 m + b\circ_2 m + m\circ_1 b + m \circ_2 b)= \\
& \pi(d(B)).
\end{align*}
 
There is a non-trivial obstruction to extend $\pi$ to arity 4 and level 2, where we have two generators $P$ and $C$: 

$\pi d(P)=0$ (no obstruction for this generator), but
$\pi d(C)$ is the sum of 
\begin{align*}
& \pi(B \circ_1 m)= 312423+314123+341243+123242+131242+131412+412434 \\
& \pi(B \circ_2 m)= 231413+214123+234143+241423+123141+414234 \\
& \pi(B \circ_3 m)=213412+234142+231342+121341+341424+313424+313234 \\
& \pi(m \circ_1 B)= 213124+231324+121314+313234 \\
& \pi(m \circ_2 B)=132423+134243+123242+142434 ,
\end{align*}
that is the sum of 24 generators, as two copies of $313234$ and $123242$ cancel out.

\begin{Lem} \label{just}
The homology class of the 2-cycle $\pi d(C) \in S_2(4)$ is $$\alpha(C)=(243)((m\circ_2 b)\circ_1 b)= [x_1,x_4][x_2,x_3] \in H_2(S(4)) \cong H_2(D_2(4)) \cong (\F_2)^{11}$$
\end{Lem}

\begin{proof}
The class $[x_1,x_4][x_2,x_3]$ is represented by the 2-cycle   
$$y= 141232+414232+141323+414323 \in S_2(4).$$ Consider in $S_3(4)$ the element 
\begin{align*}
\gamma= 2314132 + 2341432+ 2131412 +3414243+ 2131242+ 2141232 + 2313242 +\\
 2414232+  3141323 + 3414323 + 3132423 + 3134243+ 1213141 + 4142434\, .
 \end{align*}
Then $\delta(\gamma)= y-\pi d(C)$ and this proves the claim. 
\end{proof}

By lemma \ref{just}  
$$\alpha: W^2(4)=\F_2 \{ C,P\} \to H_*(S(4))$$
 is defined by $\alpha(C)=[x_1,x_4][x_2,x_3]$ and $\alpha(P)=0$.

\begin{Thm}\label{last}
The class $\alpha$ is not in the image of 
$$\partial: Hom_0(W^1,H) \to Hom_{-1}(W^2,H)$$
and therefore it represents a non trivial obstruction to the formality of $S$.

\end{Thm}
\begin{proof}
The only generators in $W^1$ that can contribute non-trivially to $Im(\partial)$ in arity $\leq 4$ are  
$A \in W_1^1(3)$ and $B \in W^1_2(3)$ that are both in arity 3.
Since
\bit
\item  $H_1(S(3))$ has a basis $\{ [x_1,x_2]x_3, [x_2,x_3]x_1 , [x_1,x_3]x_2 \}$ 
\item  $H_2(S(3))$ has a basis $\{ [[x_1,x_2],x_3], [x_1,[x_2,x_3]] \}$
\eit we have that
$Hom_0(W^1(3),H(S(3)))$ is 5-dimensional, spanned by 

\begin{align*} 
f_1:& A \mapsto m\circ_1 b=[x_1,x_2]x_3; B \mapsto 0  \\
f_2:&  A \mapsto [x_1,x_3]x_2; B \mapsto 0 \\
f_3:&  A \mapsto m \circ_2 b=[x_2,x_3]x_1 ; B \mapsto 0 \\
f_4:&  A \mapsto 0; B \mapsto b\circ_1 b= [[x_1,x_2],x_3] \\
f_5:&  A \mapsto 0; B \mapsto b\circ_2 b = [x_1,[x_2,x_3]]  
 \end{align*}
On the other hand 
$$Hom_{-1}(W^2(4),H(S(4)) = Hom(\F_2 \{ P\} , H_1(S(4))) \oplus Hom(\F_2 \{C\},H_2(S(4)))  \cong H_1(S(4)) \oplus H_2(S(4))$$ 
has dimension 6+11=17. 
By definition  \ref{part}

\begin{align*}
& \partial(f_j): P \mapsto \sum_{i=1}^2 \hat{f}_j (m) \circ_i \hat{f}_j(A) + \sum_{i=1}^3 \hat{f}_j(A) \circ_i \hat{f}_j(m) = \\
&\sum_{i=1}^2 (x_1x_2) \circ_i  f_j(A) + \sum_{i=1}^3 f_j(A) \circ_i (x_1x_2)  . \\
& \partial(f_j): C \mapsto \sum_{i=1}^2 \hat{f}_j(b) \circ_i \hat{f}_j(A) + \sum_{i=1}^3 \hat{f}_j(A) \circ_i \hat{f}_j(b) + \\
&  \sum_{i=1}^2 \hat{f}_j(m) \circ_i \hat{f}_j(B) + \sum_{i=1}^3 \hat{f}_j(B) \circ_i \hat{f}_j(m) = \\
&\sum_{i=1}^2 [x_1,x_2] \circ_i  f_j(A) + \sum_{i=1}^3 f_j(A) \circ_i [x_1,x_2]  +
\sum_{i=1}^2 (x_1x_2) \circ_i  f_j(B) + \sum_{i=1}^3 f_j(B) \circ_i (x_1 x_2) 
\end{align*}  

By substituting the values for $f_j(A)$ and $f_j(B)$ we find that

\begin{align*}
& \partial(f_1): P \mapsto [x_1,x_2]x_3x_4 ;  C \mapsto  \dots  \\
& \partial(f_2): P \mapsto [x_1,x_4]x_2x_3; C \mapsto \dots       \\
& \partial(f_3): P \mapsto [x_3,x_4]x_1x_2; C \mapsto \dots      \\
& \partial(f_4): P \mapsto 0; C \mapsto    [x_1,x_4][x_2,x_3]+[x_1,x_2][x_3,x_4]  \\
& \partial(f_5): P \mapsto 0 ; C \mapsto  [x_1,x_4][x_2,x_3]+[x_1,x_2][x_3,x_4] 
\end{align*}

The image of $C$ in the first three cases is not important since $\partial(f_j)(P)$ are linearly independent for $j=1,2,3$ but
$\alpha(P)=0$ and so if $\alpha \in Im(\partial)$ then it should be a linear combination of $\partial(f_4)$ and $\partial(f_5)$. 
However $\alpha(C)$ is not a multiple of $\partial(f_4)(C)=\partial(f_5)(C)$ and this proves that $\alpha \notin Im(\partial)$. 
\end{proof}
Theorem \ref{last} together with corollary \ref{nonformal} prove that $S$ is a non-formal planar operad over $\F_2$, and 
this proves the main theorem \ref{main}.

What happens in characteristic $p$ for $p$ odd? 
The work by Cirici-Horel \cite{CH} indicates that the obstruction defined by $\alpha$ vanishes in that case,
 and we should look for a higher obstruction. In fact we should construct the filtered model up to level $p$ in order to find an obstruction mod $p$.
 There is a striking similarity between this problem and the open problem of the formality over $\F_p$ of the little disc spaces $D_2(k)$,  homotopy equivalent to 
 the ordered configuration spaces $F_k(\R^2)$ of $k$  points in the plane, that 
 we discuss briefly at the end of the paper \cite{formality2}.

\end{document}